\documentclass[11pt]{article}
\usepackage{latexsym,amssymb}
\usepackage{amsthm}
\usepackage{amsmath}
\usepackage{graphics}
\usepackage{tikz}

\newtheorem{thm}{Theorem}
\newtheorem{lem}{Lemma}

\begin{document}
\date{\empty}

\title{Products and tensor products of graphs and homomorphisms} 

\author{Izak Broere   \\ 
Department of Mathematics and Applied Mathematics, University of Pretoria  \\
         \and
Johannes Heidema \\
Emeritus, Department of Mathematical Sciences, University of South Africa \\
}   

\maketitle 
\noindent 
Comments:   12 pages \\ 
Subjects:   Combinatorics (math.CO); Discrete Mathematics (cs.DM)  \\  
MSC-classes:  05C76 (Primary), 05C25 (Secondary)  \\
Corresponding author: Izak Broere \\ 
Email: izak.broere@up.ac.za \\
Telephone: +27-12-420-2611 \\ 
Fax: +27-12-420-3893 

\begin{abstract}  
We introduce and study, for a process P delivering edges 
on the Cartesian product of the vertex sets of a given set of graphs, 
the P-product of these graphs, thereby generalizing 
many types of product graph. 
Analogous to the notion of a multilinear map (from linear algebra), a P-morphism 
is introduced and utilised to define a P-tensor product of graphs, 
after which its uniqueness is demonstrated.  \\ 
Congruences of graphs are utilised to show a way to handle 
projections (being weak homomorphisms) in this context.     
Finally, the graph of a homomorphism and a  
P-tensor product of homomorphisms are introduced, 
studied, and linked to the P-tensor product of graphs. 
\end{abstract}

\section{Preliminaries}

\noindent 
At least half a dozen types of product graph have been introduced and studied.  
One of these types, the ``direct product'', is known under at least ten different 
names (\cite{HIK11}, p.\ 36). 
In the interest of some consolidation towards the identification 
and study of commonalities in the spectrum of possible products, we introduce the 
general notion of a ``$\mathbf{P}$-product'' of graphs, covering a number of cases.  
One of the names of the direct product is ``tensor product'', 
a well-known concept in linear algebra and category theory -- it can 
be defined by the tensor product of matrices. 
Rosinger \cite{Ro08} generalizes the notion of tensor product from 
vector spaces to structures with arbitrary binary operations, 
and even beyond those to structures with arbitrary ``generators'' 
on the underlying sets, with then binary operations as a special case 
of generators.  
We show, i.a., that the $\mathbf{P}$-product -- and hence all its 
instances -- has all the attributes required of a general ``tensor product''.  
The notion of a $\mathbf{P}$-(tensor) product and its role can be transposed from 
graphs to (sets of) graph \textit{homomorphisms}, with analogous results. 

For those notions on graphs in general not defined here, we refer the reader to \cite{Di10}.  
The Handbook of Product Graphs \cite{HIK11} is a comprehensive treatise 
on many types of product of mainly finitely many, mainly finite graphs.  
Except when explicitly stated otherwise, 
all graphs considered here are simple, undirected and unlabelled, 
and have non-empty vertex sets.   
There is, in general, no upper bound on the cardinalities of sets we use; 
neither on the vertex sets (and edge sets) of graphs we use nor on index 
sets (used, amongst others, to describe a set of graphs). 
More definitions, especially of the concepts ``loop-allowing graph'', 
``loopy graph'', and ``congruences''  
on such graphs, which are new in graph theory, will be given in Section 4 --  
they were introduced and studied in \cite{BHP15}.  

A graph $G$ with vertex set $V$ and edge set $E$ will typically be denoted by    
$G = (V, E)$; when we are dealing with different graphs, we shall use the notation 
$V_G$ for $V$ when the description of $G$ contains no subscripts, and 
$V(G)$ otherwise, similarly we shall use $E_G$ or $E(G)$ for $E$. 
A \textit{(graph) homomorphism} is an edge preserving mapping from the vertex set of a graph 
into the vertex set of a graph. 
When two graphs are isomorphic, one will be called a \textit{clone} of the other.  
We shall use the abbreviation ``iff'' for the logical connective 
``if and only if''.  

If $x$ and $y$ are elements of some set, we shall denote an ordered pair 
formed by them by $(x,y)$ and the unordered pair $\{ x,y \}$ formed by them by $xy$, 
especially when this pair is an (undirected) edge in some graph.   
We use the symbol $\bigsqcup$ for the \textit{disjoint} union of sets and graphs. 
The next definition is a standard definition from set theory. 

\textbf{The Cartesian product of sets:}  
The \textit{Cartesian product} $\prod_{i \in I} V_i$ \textit{of a collection of sets} 
$\{V_i \mid i \in I \}$ is the set of all functions $f$ from $I$ 
to $\bigsqcup_{i \in I} V_i$ which satisfy the requirement that  
$f(i) \in V_i$ for all $i \in I$.

\section{The $\mathbf{P}$-product of graphs}

\noindent 
In their approach to the classification of products, the authors of 
\cite{HIK11} require (on p.\ 41) that ``The edge set of the product 
should be determined by some definite rule.'' 
We introduce a notion of ``rule'' corresponding to theirs for our situation in this section. 
In the sequel, let ${\cal G} := \{ G_i \mid i \in I \}$ be any given set of graphs, 
and let $V := \prod_{i \in I} V(G_i)$ be the Cartesian product of their vertex sets.  

\textbf{A (product) process $\mathbf{P}$:}
Throughout this paper $\mathbf{P}$ denotes (the description or 
definition of) a process or rule that 
delivers a set of unordered pairs of vertices of $V$.  
This process must be only dependent on the index set $I$ and its elements
(including possibly set-theoretical structure on $I$, such as 
a distinguished subset of $I$ or an order relation on $I$), and the 
vertex sets $V(G_i)$ and edge sets 
$E(G_i)$ of the graphs $G_i$ and must have the property that it delivers  
a unique set of unordered pairs once ${\cal G}$ is given; this set of unordered 
pairs will be denoted by $E\mathbf{P}$. 

\textbf{The  $\mathbf{P}$-product of graphs:} 
The  $\mathbf{P}$\textit{-product of the set $\cal G$ of graphs} is the graph 
$\mathbf{P}({\cal G}) := \mathbf{P}_{i \in I} G_i : = (V, E\mathbf{P})$ 
with the Cartesian product $V$ as vertex set  
of which the edge set $E\mathbf{P}$ is this set of unordered pairs 
delivered by $\mathbf{P}$. 

\textbf{Examples of $\mathbf{P}$-products:}   
In each example below we assume that a set $\cal G$ of graphs is given 
and we describe a process $\mathbf{P}$ that leads in each of the first four cases 
to a well-known product of graphs.   
In \cite{HIK11} (p.\ 35) the products in the first three 
examples are called the ``three fundamental products'' --   
the reader is invited to check that our terminology 
is in line with theirs by applying our definitions to that situation.   \\ 
$\bullet$ 
\textbf{The Cartesian product of graphs:}  
For the \textit{Cartesian product} {\raisebox{-0.56mm}{\Large \bf $\Box$}}$_{i \in I} G_i$ 
of $\cal G$ (or {\raisebox{-0.56mm}{\Large \bf $\Box$}}${\cal G}$) 
we stipulate for the process $\mathbf{P}$ that two vertices $f$ and $g$
are adjacent in {\raisebox{-0.56mm}{\Large \bf $\Box$}}$_{i \in I} G_i$ 
iff there is an index $j \in I$ (in general depending on $f$ and $g$) 
such that $f(j)g(j) \in E(G_j)$ while $f(i) = g(i)$ for all $i \in I$ with $i \neq j$.    \\ 
$\bullet$  
\textbf{The direct product of graphs:}  
For the \textit{direct product} {\Large$\times$}$_{i \in I} G_i$ 
(or {\Large$\times$}${\cal G}$) 
we stipulate for the process $\mathbf{P}$ that two vertices $f$ and $g$
are adjacent in {\Large$\times$}$_{i \in I} G_i$ iff for every index $i \in I$ 
we have that $f(i)g(i) \in E(G_i)$.  
The edge set of the direct product {\Large$\times$}$_{i \in I} G_i$ 
will also be denoted by $E_{\times}$.   \\ 
$\bullet$  
\textbf{The strong product of graphs:}  
For the \textit{strong product} 
{\raisebox{-.3ex}{\Large\bf$\Box$}\hspace{-3.9mm}{\Large${\times}$}}$_{i \in I} G_i$ 
(or {\raisebox{-.3ex}{\Large\bf$\Box$}\hspace{-3.9mm}{\Large${\times}$}}${\cal G}$) of $\cal G$ 
we stipulate for the process $\mathbf{P}$ that two vertices $f$ and $g$
are adjacent in 
\mbox{{\raisebox{-.3ex}{\Large\bf$\Box$}\hspace{-3.9mm}{\Large${\times}$}}$_{i \in I} G_i$} 
iff there is a proper subset $K \subset I$ 
(in general depending on $f$ and $g$) such that for all 
$j \in I \setminus K$, 
$f(j)g(j) \in E(G_j)$, while for all $k \in K$, $f(k) = g(k)$.    \\ 
$\bullet$  
\textbf{The lexicographic product of graphs:}   
For this type of product the process $\mathbf{P}$ is co-determined by 
an order on the factors $G_i$.  
We assume that $\langle I, < \rangle$ is well-ordered, 
so is order isomorphic to some fixed ordinal $\langle \kappa, \in \rangle$.  
For the \textit{lexicographic product} $\bigcirc_{i \in I} G_i$ 
(or $\bigcirc {\cal G}$) of $\cal G$ we then stipulate for the process 
$\mathbf{P}$ that two vertices $f$ and $g$ are adjacent in $\bigcirc_{i \in I} G_i$ 
iff $J := \{ j \in I \mid f(j)g(j) \in E(G_j) \} \neq \emptyset$ 
and, when $m$ is then the $<$-minimum element of $J$, 
for every $k \in I$ such that $k < m$ (if such exist) $f(k) = g(k)$.      \\
$\bullet$  
\textbf{The $D$-product of graphs:} 
This is new product, as far as we know.  
Let $D$ be a fixed (distinguished) non-empty subset of the index set $I$.  
For the \textit{$D$-product} {\large{$\bigcirc$}}$\!\!\!\!\!\! {\mathbf D}_{i \in I} G_i$ 
(or {\large{$\bigcirc$}}$\!\!\!\!\!\! {\mathbf D} \, \, \cal G$)   
we stipulate for the process $\mathbf{P}$ that two vertices $f$ and $g$
are adjacent in {\large{$\bigcirc$}}$\!\!\!\!\!\! {\mathbf D}_{i \in I} G_i$ iff 
for every index $j \in D$ we have $f(j)g(j) \in E(G_j)$.  

\textbf{The triples of sets $J(fg), K(fg)$ and $L(fg)$:}  
Given a process $\mathbf{P}$ delivering the edge set $E\mathbf{P}$ 
of ${\mathbf{P}(\cal G})$ for a set of graphs $\cal G$, 
every unordered pair $\{ f, g \}$ (with $f,g \in V$) and 
in particular every edge $fg \in E\mathbf{P}$ determines the 
following three subsets of the index set $I$:  \\ 
$J(fg) : = \{ j \in I \mid f(j)g(j) \in E(G_j)  \}$;  \\ 
$K(fg) : = \{ k \in I \mid f(k) = g(k) \}$; and \\ 
$L(fg) : = \{ l \in I \mid f(l)g(l) \not\in E(G_l) \mbox{ and } f(l) \neq g(l)\} 
= I \setminus (J(fg) \cup K(fg))$.  

Conversely, if for any vertices $f,g \in V$ we know 
that  $fg \in E\mathbf{P}$ and we know the three sets of 
indices $J(fg)$, $K(fg)$, and $L(fg)$ -- even the first two 
will do -- then we know for every $i \in I$ whether 
$f(i)g(i) \in E(G_i)$, or $f(i) = g(i)$, or neither.  
Different choices of processes $\mathbf P$, leading to different decisions on 
whether for $f,g \in V$ we have $fg \in E\mathbf{P}$ or not, 
may entail different set-theoretical constraints on the sets 
$J(fg)$, $K(fg)$, and $L(fg)$ -- constraints on 
their cardinalities, on whether they are empty or not, 
constraints of inclusion, etc.  
For a given $\mathbf P$ we call these constraints the \textit{$\mathbf P$-constraints} on 
the three index sets.  
For our five examples of $\mathbf P$-products the $\mathbf P$-constraints are, 
respectively, the following for every edge $fg \in E{\mathbf P}$:  \\
{\raisebox{-0.56mm}{\Large \bf $\Box$}}: For the Cartesian product   \\  
\hspace*{.65cm} $|J(fg)| = 1$, i.e., $J(fg)$ is a singleton, say $\{ j \} \subseteq I$;    \\
\hspace*{.65cm} $K(fg) = I \setminus J(fg)$, say $I \setminus \{ j \}$; and    \\ 
\hspace*{.65cm} $L(fg) = \emptyset $.   \\ 
{\Large$\times$}: For the direct product   \\  
\hspace*{.65cm} $J(fg) = I$;  $K(fg) = L(fg) = \emptyset$.     \\ 
{\raisebox{-.3ex}{\Large\bf$\Box$}\hspace{-3.9mm}{\Large${\times}$}}:  
For the strong product  \\
\hspace*{.65cm} $J(fg) \neq \emptyset$;    \\
\hspace*{.65cm} $K(fg) = I \setminus J(fg)$; and    \\ 
\hspace*{.65cm} $L(fg) = \emptyset$.    \\  
$\bigcirc$:  For the lexicographic product the well-ordering 
$\langle I,< \rangle$ is fixed once and for all, and   \\ 
\hspace*{.65cm} $J(fg) \neq \emptyset$;   \\
\hspace*{.65cm} $K(fg) \supseteq \{ k \in I \mid k < m \}$, where  $m$ is 
the $<$-minimum of $J(fg)$; and  \\ 
\hspace*{.65cm} $L(fg) = I \setminus (J(fg) \cup K(fg))$.     \\ 
{\large{$\bigcirc$}}$\!\!\!\!\!\! {\mathbf D}$ :  For the $D$-product 
the non-empty set $D \,\,(\subseteq I)$  is fixed once and for all, and \\ 
\hspace*{.65cm} $J(fg) \supseteq D$;   \\
\hspace*{.65cm} $K(fg) \subseteq I \setminus J(fg)$; and    \\ 
\hspace*{.65cm} $L(fg) = I \setminus (J(fg) \cup K(fg))$.       
 
Furthermore, in all five these examples, $J(fg)$, $K(fg)$, and $L(fg)$ 
are pairwise disjoint (of course by definition) for every $fg \in E\mathbf{P}$.  
In fact, in all of these examples $\mathbf{P}$ induces ``$JKL$'' 
that satisfy the following \textit{universal constraints}:  \\ 
For all $fg \in E\mathbf{P}$, $J(fg)$, $K(fg)$, and $L(fg)$  are pairwise 
disjoint, their union is $I$, and $J(fg) \neq \emptyset $.  

In \cite{HIK11} (pp.\ 41 -- 43) an \textit{incidence function} 
$\delta$ is introduced for any graph $G$ (not necessarily 
a product graph) as $\delta: V_G \times V_G \rightarrow \{ 1, \Delta ,0 \}$.  
The incidence functions of two graphs then determine the incidence 
function -- and hence the adjacency relation -- on their product of a certain type.  
It should be clear that the three values of $\delta$, 
$\delta(g,g') = 1$ if $g \neq g'$ and $gg' \in E_G$; 
$\delta(g,g') = \Delta$ if $g = g'$; and 
$\delta(g,g') = 0$ if $g \neq g'$ and $gg' \not\in E_G$, 
function analogously to, respectively, our three subsets 
$J(fg)$; $K(fg)$; and $L(fg)$ of $I$ in determining 
the adjacency relation in a $\mathbf P$-product 
of any number of graphs.  

We have seen that every product process $\mathbf P$ induces a triple of functions 
$$ J, K, L : E{\mathbf P} \rightarrow {\cal P}(I); fg \mapsto J(fg), K(fg), L(fg)$$ 
such that for every $fg \in E{\mathbf P}$, the triple $J(fg)$, $K(fg)$, and $L(fg)$  
satisfies the $\mathbf P$-constraints.  
It may happen that, conversely, the existence of $J, K, L$ triples 
of subsets of $I$ satisfying certain constraints 
completely determines $E{\mathbf P}$.  

{\bf Constraint-determined product processes:}
We shall call the product process $\mathbf P$ \textit{constraint-determined} 
iff the following holds:  \\ 
There exists a set $C$ of constraints of the relevant nature 
(including the universal ones) on triples 
$J, K, L : \{ \{ f, g \} \mid f, g \in V \mbox{ and } f \neq g \} \rightarrow {\cal P}(I)$ 
such that for every unordered pair $\{ f, g \} , f \neq g$, of elements 
of $V$, $fg \in E{\mathbf P}$ iff there exists 
some triple $J\{ f,g \}, K\{ f,g \},L\{ f,g\}$ of subsets of $I$  
satisfying $C$.  
In this case we could then legitimately consider $C$ to be (equivalent to) 
the $\mathbf P$-constraints.  
Careful examination of their definitions and our exposition of the constraints 
for our five examples of product processes verifies the next result.   
\vspace{.165cm}

\begin{lem}\label{lemma1} 
All the product processes {\raisebox{-0.56mm}{\Large \bf $\Box$}}, 
{\Large$\times$}, {\raisebox{-.3ex}{\Large\bf$\Box$}\hspace{-3.9mm}{\Large${\times}$}}, 
$\bigcirc$,  and {\large{$\bigcirc$}}$\!\!\!\!\!\! {\mathbf D}$
are constraint-determined.  
\end{lem} 
\vspace{.43cm} 

{\bf Permutable product processes:}
Beyond the attribute that a $\mathbf P$ may have 
of being constraint-determined, we now want 
to define a more general attribute of $\mathbf P$ being ``permutable''.  
The process $\mathbf P$ is \textit{permutable} iff the following holds:  \\ 
Whenever $p : I \rightarrow I$ is a permutation (bijection) of $I$  
that respects whatever occasional fixed set-theoretical structure 
$\mathbf P$ has imposed upon $I$ (as for $\bigcirc$ and 
{\large{$\bigcirc$}}$\!\!\!\!\!\! {\mathbf D}$), 
then, for all $f,g \in V$, $p(f)p(g) \in E{\mathbf P}$ iff 
$fg \in E{\mathbf P}$.  
(Any permutation $p : I \rightarrow I$ induces a 
bijection again (naughtily) called  $p : V \rightarrow V$; $f \mapsto p(f)$; 
$(p(f))(i) = f(p(i))$.)
The class of permutable processes includes the class of constraint-determined 
processes:  
\vspace{.165cm}

\begin{lem}\label{lemma2}  
If product process $\mathbf P$ is constraint-determined, then $\mathbf P$ 
is permutable.  
\end{lem} 

\begin{proof}
Consider a constraint-determined $\mathbf P$ and any 
permutation $p : I \rightarrow I$ respecting occasional $\mathbf P$-imposed 
structure on $I$.  
Then $p$ preserves every $\mathbf P$-constraint in the following sense:  \\ 
Every set-theoretical constraint on one, two, 
or three subsets of $I$ (like emptiness, non-emptiness, 
cardinality, set-theoretical difference, inclusion, etc.)\
holds intact between the $p$-images of those sets. 
Since these constraints determine adjacency in $E\mathbf P$  
by the \textit{existence} of $JKL$ triples, the 
$p$-images of those triples determine \textit{the same} $E\mathbf P$.  
So the bijection $p : V \rightarrow V$ establishes the isomorphism 
${\mathbf P}_{i \in I} G_i \cong {\mathbf P}_{i \in I} G_{p(i)}$.
\end{proof}
\vspace{.165cm} 

An immediate consequence is the following.  
\vspace{.165cm}

\begin{lem}\label{lemma3}  
Any graph product yielded by a permutable (and in particular by a 
constraint-determined) process is commutative and associative 
in every possible sense of those attributes.  
\end{lem}
\vspace{.43cm}  

Occasional provisos should not be forgotten.  
In the definition of the lexicographic product a fixed well-ordering 
$\langle I, < \rangle$ is assumed and the only permutation of 
$I$ preserving this order is $id_I$, the identity function 
on the set $I$. 
Similarly, for the $D$-product only permutations of $I$ 
that map the fixed distinguished set $D$ of indices 
onto itself should be allowed.

\section{$\mathbf P$-morphisms and the $\mathbf{P}$-tensor product of graphs}

\noindent 
In order to advance our study of the $\mathbf{P}$-product of graphs, 
we need the notion of a morphism from the set $V$ to a graph $H$ 
that is linked to the process $\mathbf{P}$. 

\textbf{$\mathbf{P}$-morphisms:}  
Whenever we have a set of graphs ${\cal G} = \{ G_i \mid i \in I \}$, 
a process $\mathbf{P}$ (delivering 
${\mathbf{P}(\cal G}) = (V, E\mathbf{P})$), and a graph $H$ available, we define  
a $\mathbf{P}$-\textit{morphism} to the graph $H$ as a function 
$\delta: V \rightarrow V_H$ such that $\delta$ is a homomorphism 
$\delta: {\mathbf{P}(\cal G}) \rightarrow H$; we may write 
$\delta: V \stackrel{\mathbf{P}}{\rightarrow} H$ to indicate this fact.   
We note already here that $\mathbf{P}$-morphisms in graph theory are 
analogous to \textit{multilinear mappings} in linear algebra. 
In Section 7 we shall clarify this remark. 

Next we define a tensor product of a set of graphs which is also  
linked to a given process $\mathbf{P}$.  
This definition follows the idea found in many algebra textbooks, 
especially in the context of the linear algebra of vector spaces; 
the one we follow in particular is the so-called ``universal property 
of tensor products'' given in Theorem 14.3 of \cite{Ro92}, which is there shown 
to be (for the case of two graphs) equivalent to 
the definition, formulated in terms of free vector spaces 
and bilinear maps, given on p.\ 298 of \cite{Ro92}. 

\textbf{The $\mathbf{P}$-tensor product of graphs:}  
Assume as given a process $\mathbf{P}$ as above.  
For any given set of graphs $\cal G$ 
(so that the graph ${\mathbf{P}(\cal G})$ is 
uniquely determined), as well as some 
graph $T$ together with a fixed $\mathbf{P}$-morphism 
$\varphi: V \stackrel{\mathbf{P}}{\rightarrow} T$ 
(delivering the homomorphism $\varphi: {\mathbf{P}(\cal G}) \rightarrow T$),  
the pair $(\varphi, T)$ -- or just $T$ when $\varphi$ is 
understood --  is called a $\mathbf{P}$\textit{-tensor product of} $\cal G$ 
when the following holds:  \\ 
(i)  $\varphi$ is surjective, i.e., $\varphi(V) = V_T$; and  \\ 
(ii) if $H$ is any graph and $\delta: V \stackrel{\mathbf{P}}{\rightarrow} H$ 
is any $\mathbf{P}$-morphism, then there exists a homomorphism 
$\delta^\star: T \rightarrow H$ such that $\delta = \delta^\star \circ \varphi$. 

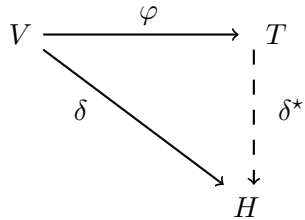
\begin{figure}[ht]
\label{fig:network}
\begin{center}
\begin{tikzpicture}
	\draw [->, thick] (3.7,4.2) -- (6.3,4.2);
	\draw [->, thick] (3.7,4) -- (6.1,2.2);
	\draw [-, thick] (6.5,4.0) -- (6.5,3.8);
	\draw [-, thick] (6.5,3.6) -- (6.5,3.4);
    \draw [-, thick] (6.5,3.2) -- (6.5,3.0);
	\draw [-, thick] (6.5,2.8) -- (6.5,2.6); 
	\draw [->, thick] (6.5,2.4) -- (6.5,2.2);
	\node  [black] at (3.4,4.2) {$V$};
	\node  [black] at (6.8,4.2) {$T$};
	\node  [black] at (5.1,4.45) {$\varphi$};
	\node  [black] at (7,3.2) {$\delta^\star$};
	\node  [black] at (6.4,1.9) {$H$};
	\node  [black] at (4.2,3.2) {$\delta$};
\end{tikzpicture}
\end{center}
\end{figure}
\noindent 
One may well refer to the requirements in this definition as the  
\textit{universal factorization condition of $\delta$ through $T$}.  
These morphisms, and the fact that $\delta = \delta^\star \circ \varphi$, 
are illustrated in the accompanying  commutative diagram.  

It is easy to see that conditions (i) and (ii) of the definition of 
the $\mathbf{P}$-tensor product of graphs are trivially met for each 
set of graphs $\cal G$ and each process $\mathbf P$ by the  
pair $(id_V, {\mathbf{P}(\cal G}))$ (since condition (ii) will then 
be satisfied by choosing $\delta^\star$ as $\delta$); 
this of course includes all the examples 
in the list in the previous section.     
Hence each pair $(id_V, {\mathbf{P}(\cal G}))$ described 
in that list is a $\mathbf{P}$-tensor product 
for the process $\mathbf{P}$ chosen in the example.  
We now proceed to show that this is effectively the only type 
of example of a $\mathbf{P}$-tensor product of graphs, i.e., 
we show that $(id_V, {\mathbf{P}(\cal G}))$ is up to isomorphism the only 
$\mathbf{P}$-tensor product of $\cal G$. 
This will justify us in calling $(id_V, {\mathbf{P}(\cal G}))$ 
the \textit{canonical} $\mathbf{P}$-tensor product of $\cal G$.  

The next lemma uses notation anticipating its later employment. 
\vspace{.165cm}

\begin{lem}\label{lemma2a}  
Consider two graphs $H_1 = (V_1, E_1)$ and $H_2 = (V_2, E_2)$ and  
two homomorphisms $\delta_1^\star: H_1 \rightarrow H_2$ and 
$\delta_2^\star: H_2 \rightarrow H_1$ 
such that $\delta_2^\star \circ \delta_1^\star = id_{V_1}$ 
and $\delta_1^\star \circ \delta_2^\star = id_{V_2}$.  
Then $\delta_1^\star$ and $\delta_2^\star$ are both bijective 
isomorphisms (and inverses of each other), making $H_1$ and $H_2$ 
clones of each other.  
\end{lem} 

\begin{proof}
We first prove that $\delta_1^\star$ and $\delta_2^\star$ 
are injective,  and surjective onto $H_2$ and $H_1$ respectively.  
The fact that $\delta_2^\star \circ \delta_1^\star = id_{V_1}$, 
implies that $\delta_1^\star$ is injective.  
(Suppose not, and that $v_1, v_1' \in V_1$, $v_1 \neq v_1'$, while 
$\delta_1^\star(v_1) = \delta_1^\star(v_1') : = v_2 \in V_2$.  
Then $\delta_2^\star(v_2)$ has to be both $v_1$ and $v_1'$, 
which is impossible.) 
By the symmetry of the conditions on $\delta_1^\star$ and on 
$\delta_2^\star$, $\delta_2^\star$ is also injective.  

Next we show that $\delta_1^\star$ is surjective onto $V_2$.  
Suppose not, and that there is a $v_2 \in V_2$ such that for all 
$v_1 \in V_1$, $\delta_1^\star(v_1) \neq v_2$, while 
$\delta_2^\star(v_2) := v_1' \in V_1$.  
Then $\delta_1^\star(v_1') = (\delta_1^\star \circ \delta_2^\star)(v_2) = v_2$, 
a contradiction.  
Similarly, $\delta_2^\star$ is surjective onto $V_1$.  
So, both $\delta_1^\star$ and $\delta_2^\star$ are bijective homomorphisms.  

The homomorphism $\delta_1^\star$ (and by symmetry also $\delta_2^\star$) 
is an isomorphism, since it preserves not only adjacency, but also non-adjacency.  
Suppose $v_1, v_1' \in V_1$ with $v_1v_1' \not\in E_1$.  
Then $\delta_1^\star(v_1)\delta_1^\star(v_1') \not\in E_2$, for were 
$\delta_1^\star(v_1)\delta_1^\star(v_1') \in E_2$, then (since $\delta_2^\star$ 
is a homomorphism) we would have \\ 
$[(\delta_2^\star\circ\delta_1^\star)(v_1)][(\delta_2^\star\circ\delta_1^\star)(v_1')] 
= v_1v_1' \in E_1$.
\end{proof}
\vspace{.165cm} 

\noindent 
The result of Lemma \ref{lemma2a} is useful in the proof of Theorem \ref{Thm1}.  

We now show, for a given process $\mathbf{P}$ and a given 
set ${\cal G}$ of graphs, that $(id_V, {\mathbf P}({\cal G}))$ is ``up to isomorphism'' 
the only pair $(\varphi , T)$ which has the properties 
required by the definition of a $\mathbf{P}$-tensor product of these graphs.  
The exact meaning of the phrase ``up to isomorphism'' will be strengthened by 
a discussion after the proof of the theorem.   
\vspace{.165cm}  
 
\begin{thm}\label{Thm1}   
$(id_V, {\mathbf P}({\cal G}))$ is (up to isomorphism)  
the unique $\mathbf{P}$-tensor product of ${\cal G}$.    
\end{thm} 

\begin{proof}
We have already remarked that conditions (i) and (ii) of the definition of 
the $\mathbf{P}$-tensor product of ${\cal G}$ are trivially met 
by the pair $(id_V, {\mathbf P}({\cal G}))$ by choosing $\delta^\star$ as $\delta$.  

To show uniqueness, assume that $(\varphi , T)$ is any $\mathbf{P}$-tensor product of ${\cal G}$. 
We first apply the definition of ``$(\varphi , T)$ is a $\mathbf{P}$-tensor product 
of $\cal G$'' and use in it the graph $\mathbf{P}({\cal G})$ for $H$ 
and the identity map $id_V$ for $\delta$ (which is a $\mathbf{P}$-morphism)  
to conclude that there exists a homomorphism $id_V^\star : T \rightarrow \mathbf{P}({\cal G})$ 
such that 
\begin{equation}
id_V = id_V^\star \circ \varphi.
\end{equation}

On the other hand, since $\varphi : V \rightarrow T$ is a $\mathbf{P}$-morphism 
and $(id_V, {\mathbf P}({\cal G}))$ is a $\mathbf{P}$-tensor 
product of ${\cal G}$, we can apply the definition again (with $T$ for $H$) 
to conclude that there exists a homomorphism $\varphi^\star : \mathbf{P}({\cal G}) \rightarrow T$ 
such that $\varphi = \varphi^\star \circ id_V$; so 
\begin{equation}
\varphi = \varphi^\star
\end{equation}
since $id_V$ is an identity map.  
By equations (1) and (2), it now follows that 
\begin{equation}
id_V = id_V^\star \circ \varphi^\star.
\end{equation}  
  
Now consider any $y \in V_T$.  
Since $\varphi$ is surjective, there exists at least one $x \in V$ with $\varphi(x) = y$.  
Then $id_V^\star (y) \in V$ and, since $id_V$ is an identity map, 
$id_V^\star (y) = id_V(id_V^\star (y))$.  
Furthermore, 
\begin{eqnarray} 
\varphi^\star(id_V^\star(\varphi(x))) & = & \varphi^\star(id_V(x)) \mbox { by (1)}  \nonumber \\
                                         & = & \varphi^\star(x)  \mbox{ since $id_V$ is an identity map} \nonumber \\  
                                         & = & \varphi(x)  \mbox{ by (2),} \nonumber
\end{eqnarray}
so that $\varphi^\star(id_V^\star(y)) = y$, for all $y \in V_T$.  
Hence 
\begin{equation}
id _{V_T} = \varphi^\star \circ id_V^\star.
\end{equation}
It is now clear that $\varphi^\star$ and $id_V^\star$ satisfy the premises 
of Lemma \ref{lemma2a} for $\delta_1^\star$ and $\delta_2^\star$ 
respectively and hence we can conclude that    
$T$ is a clone of ${\mathbf P}({\cal G})$.  
\end{proof}

\noindent 
We now return, as promised, to the phrase ``up to isomorphism''.   
We have shown in the above proof, for any pair 
$(\varphi , T)$ which is a $\mathbf{P}$-tensor product of ${\cal G}$, 
that $T$ is a clone of ${\mathbf P}({\cal G})$, i.e., $T$ looks like 
${\mathbf P}({\cal G})$ ``up to an isomorphism''.  
But more than that is true: 
The action of the $\mathbf{P}$-morphism $\varphi$ is also imitating the action 
of the identity map $id_V$ through equation (1) since $id_V^\star$ 
is an isomorphism, i.e., $\varphi$ (from $V$ to $V_T$) acts like an identity map on 
$V$ ``up to the isomorphism'' $id_V^\star$ (taking $V_T$ back to $V$). 

These results justify that the $\mathbf P$-tensor product 
$(id_V, {\mathbf P}({\cal G}))$ be denoted as 
$\mathbf{\large{\bigcirc\!\!\!\!\! \times}}_{\mathbf P}{\cal G}$.

\section{Making each $\pi_i$ a homomorphism using congruences}

\noindent 
In this section we use the theory of congruences on graphs developed in \cite{BHP15}. 
This development is given there in full detail for simple graphs, while the last section 
is devoted to graphs which have a loop at every vertex, called \textit{loopy} graphs there. 
A \textit{loop-allowing} (hence generally non-simple) graph allows 
loops, but does not (like a loopy graph) prescribe them at every vertex. 
We remark that the definition of a congruence of a loop-allowing graph 
(to follow) is simpler than that of a simple graph and also 
from that of a loopy graph. 

While in the rest of this article all graphs are simple, 
in this section and at the end of the next section only, 
we shall allow 
a graph construction (defined on all loop-allowing 
graphs) which, when applied to a simple graph, 
may yield a loop-allowing graph as a result.  
To describe this construction -- forming the quotient of 
a graph modulo a congruence -- we first  
define the notion of ``congruence'' on a loop-allowing graph.  

{\bf Congruences and quotients:} 
A \textit{congruence on} a loop-allowing graph $G = (V, E)$ is a pair 
$\theta = (\sim, \widehat{E})$ such that \\
(i)  $\sim$ is an equivalence relation on $V$ (hence $id_{V} \subseteq \, \, \sim$);   \\  
(ii) $\widehat{E}$ is a set of unordered pairs of elements 
from $V$ with $E \subseteq \widehat{E}$; and  \\ 
(iii) when $x, y, x', y' \in V$, $x \sim x'$, $y \sim y'$, and 
$xy \in \widehat{E}$, then $x'y' \in \widehat{E}$, 
i.e., $\widehat{E}$ is substitutive with respect to $\sim$.   \\ 
The congruence $\iota_G := (id_{V}, E)$ on $G$ is the smallest congruence on $G$, 
i.e., $\subseteq$-smallest in both components.

Given any congruence $\theta = (\sim, \widehat{E})$ on a (loop-allowing) 
graph $G = (V, E)$, we define a new graph, denoted by $G/\theta$ and called 
\textit{(the quotient of) $G$ modulo} $\theta$, as follows:   \vspace{-2mm}
\begin{eqnarray} 
G / \theta &  : = &  (V_{G/\theta}, E_{G/\theta})  \nonumber  \\ 
           &  :=  &  (V / \sim, \{ [x][y] \mid xy \in \widehat{E} \}),  \nonumber 
\end{eqnarray} 
where $[x] \in V_{G/\theta}$ denotes the $\sim$-equivalence class of $x \in V$.  
   
We note that the surjective mapping $x \mapsto [x]$ from $V_{G}$ onto $V_{G / \theta}$ 
establishes the \textit{natural} or \textit{canonical} homomorphism $G \rightarrow G / \theta$.  
When $\theta = \iota_G$, this natural homomorphism is the mapping $x \mapsto \{ x \}$, $x \in V_G$ 
and it is an isomorphism, i.e., $G \cong G / \iota$.  
Making the distinction between $x$ and $\{ x \}$ in such a case 
is so superficial that we shall not always bother to do so and thus treat 
them as if they are equal; using the phrase 
``identity function'' in the last result in this section is the first 
instance where this remark is applied fully.   

Given any graph homomorphism between loop-allowing graphs, 
say $\varphi : G \rightarrow H$, we define 
a congruence on $G$, denoted by $\theta_\varphi$ and called the 
\textit{congruence induced by} $\varphi$ or the \textit{kernel of} $\varphi$, by \vspace{-2mm}
\begin{eqnarray}
 \theta_\varphi & := & (\sim_\varphi, \widehat{E_\varphi})  \nonumber   \\ 
  & := & (\{ (x,y) \in V_G^2 \mid \varphi(x) = \varphi(y) \},\{ uv \mid u,v \in V_G 
  \mbox{ and } \varphi(u)\varphi(v) \in E_H \}) \nonumber   \\ 
  & =  & (\varphi^{-1}[id_{\varphi(V_G)}], \varphi^{-1}[E(H[\varphi(V_G)])]).   \nonumber \vspace{-2mm}
\end{eqnarray} 
It should be immediately clear that $\theta_\varphi$ is a congruence on $G$.  

Now back to $\mathbf{P}$-products: 
Assume as given the set ${\cal G} = \{ G_i \mid i \in I \}$ of 
loop-allowing graphs with $G_i = (V_i,E_i)$ 
for each $i \in I$ and a process ${\mathbf P}$.  
Consider the $\mathbf{P}$-tensor product ${\mathbf P}({\cal G})$ 
of ${\cal G}$, any $i \in I$, and the \textit{projection} $\pi_i: V \rightarrow V_i$ which, 
for each $f \in V$ maps $f$ to $f(i)$. 
We remark that, for the direct product {\Large$\times$}${\cal G}$, the projection 
$\pi_i$ is a homomorphism $\pi_i :${\Large$\times$}${\cal G} \rightarrow G_i$
for every $i \in I$,  but 
for other products this does not hold in general.  
Can we somehow, for every  ${\mathbf P}$-product, restore each $\pi_i$ 
to its status as a homomorphism by transforming $G_i$ in a uniform way?  
We shall now demonstrate how we may, by taking a suitable quotient graph, 
to be called ${\mathbf P}(G_i)$, indeed reach the conclusion that 
$\pi_i : {\mathbf P}({\cal G}) \rightarrow {\mathbf P}(G_i)$ is a homomorphism.  
Let
\begin{eqnarray} 
\theta_i   &  : = &  (id_{V_i}, \widehat{E_i}), \mbox{ where } \nonumber \\ 
\widehat{E_i} &  : = &   \{ xy \mid x,y \in V_i, \mbox{ and } (xy \in E_i \mbox{ or there exist } 
                    f,g \in V \mbox{ such that }  \nonumber \\  
              &       & \,\,\,\,\,\,\,\,\,\,\,\,\,\,\, fg \in E{\mathbf P}, 
                        f(i) = x,  \mbox{ and } g(i) = y) \}.  \nonumber   
\end{eqnarray}              
It is easy to see that $\theta_i$ is a congruence on $G_i$ for each $i \in I$ 
(according to the above definition).  
With respect to this congruence, we now define the quotient graph 
${\mathbf P}(G_i) : = (V_i, \widehat{E_i}) = G_i / \theta_i$.
Hence ${\mathbf P}(G_i)$ is a well-defined graph -- 
but it is loop-allowing (even if $G_i$ is simple)  
since $\mathbf P$ may be such that $fg \in E{\mathbf P}$ while 
$x = f(i) = g(i) = y$, giving $xx \in \widehat{E_i}$ 
and hence $xx = \{ x \} \{ x \} = [x][x] \in E(G_i / \theta_i)$.   
This is our way to surmount a problem that Hammack, Imrich and Klav\v{z}ar handle in \cite{HIK11} 
by introducing ``weak homomorphisms'' (functions allowed to map adjacent 
vertices to the same vertex) for other than direct products.   
The next result is now trivial.  
\vspace{.165cm}  
 
\begin{lem}\label{ProjHom}  
Every projection $\pi_i : V \rightarrow V_i$ is a homomorphism 
$\pi_i : {\mathbf P}({\cal G}) \rightarrow {\mathbf P}({\cal G}_i)$.  
For the special case of the direct product the allowance of loops in 
$\widehat{E_i}$ can be removed as unnecessary, with $\widehat{E_i} = E_i$,  
$\theta_i$ the identity congruence $\iota_{G_i}$ on $G_i$, 
and ${\mathbf P}(G_i) = G_i$.   
\end{lem} 
\vspace{.43cm} 

\noindent 
This Lemma immediately entails the next result, 
which is trivial when ${\mathbf P} = ${\Large $\times$}.  
\vspace{.165cm}  

\begin{thm}\label{id=homo}  
For all sets $\cal G$ of graphs and all processes $ {\mathbf P}$ to construct 
their $ {\mathbf P}$-product $ {\mathbf P}({\cal G})$, the 
identity function on $V$  
is a bijective homomorphism from ${\mathbf P}({\cal G})$ to the direct product 
{\Large$\times$}$_{i \in I} {\mathbf P}(G_i)$  of the loop-allowing graphs 
${\mathbf P}(G_i)$, i.e., ${\mathbf P}({\cal G}) \subseteq$ 
{\Large$\times$}$_{i \in I} {\mathbf P}(G_i)$.
\end{thm} 

\begin{proof} 
If $fg \in E{\mathbf P}$, then, (by the definition of adjacency in 
${\mathbf P}(G_i)$ and Lemma \ref{ProjHom}), $f(i)g(i) \in \widehat{E_i}$ for every 
$i \in I$, and hence $fg$ is an edge of the direct product 
{\Large$\times$}$_{i \in I} {\mathbf P}(G_i)$ of the loop-allowing 
graphs ${\mathbf P}(G_i)$. 
\end{proof}

\section{$\mathbf{P}$-tensor products of homomorphisms}

\noindent 
Beyond the furtive entry of loop-allowing 
graphs in the previous section, we now, 
until near the end of this section, restrict ourselves to simple graphs.  
Remember that the Cartesian product $X := \Pi_{i \in I}X_i$ of sets 
is the set consisting of all functions $f: I \rightarrow \bigsqcup_{i \in I} X_i$ 
satisfying $f(i) \in X_i$ for all $i \in I$ and that  
the $i$'th projection $\pi_i$ is the function 
$\pi_i : X \rightarrow X_i$ with $\pi_i(f) = f(i)$ for each $i \in I$.
Hence if, for a function $f: I \rightarrow \bigsqcup_{i \in I} X_i$, 
the value of $\pi_i(f)$ is given for each $i \in I$, then the function 
is uniquely determined and if, furthermore, this value is in $X_i$ for 
each $i \in I$, then this function $f$ is in the Cartesian product $\Pi_{i \in I}X_i$.

We now assume as given a product process $\mathbf{P}$, a fixed index set $I$,   
two sets of graphs ${\cal G} = \{ G_i \mid i \in I \}$ and ${\cal H} = \{ H_i \mid i \in I \}$,  
as well as a set $\Phi = \{ \varphi_i \mid i \in I \}$ of functions in which, 
for each $i \in I$, $\varphi_i : V(G_i) \rightarrow V(H_i)$. 
For ease of notation we let  
$V_i := V(G_i)$, $W_i := V(H_i)$, for each $i \in I$, and, for the 
Cartesian products of these vertex sets,  
$V := \Pi_{i \in I}V_i$ while $W := \Pi_{i \in I} W_i$.  

Note that, for the given set $\Phi$ of functions, there exists a unique function  
$\varphi : V \rightarrow W$ such that, for each $f \in V$ and each $i \in I$, 
$\pi_i(\varphi(f)) = \varphi_i(\pi_i(f))$;  
this equation can also be written (and will be utilised) as 
$\varphi(f)(i) = \varphi_i(f(i))$, thereby mitigating the venial sin 
of abusing the name $\pi_i$ for two different projections.  
The fact that there is such a function, and its  uniqueness,  
follows from the remark that its domain and function values are completely 
specified by the above defining conditions.  
The function $\varphi$ and its properties are depicted in the accompanying 
commutative diagram.

\begin{figure}[ht]
\label{fig:network}
\begin{center}
\begin{tikzpicture}
	\draw [->, thick] (4.1,4.2) -- (6.1,4.2);
	\draw [->, thick] (6.5,4.0) -- (6.5,2.2);
    \draw [->, thick] (3.7,4.0) -- (3.7,2.2);
    \draw [->, thick] (4.1,2.0) -- (6.1,2.0);
	\node  [black] at (3.4,4.2) {$V$};
	\node  [black] at (3.4,1.9) {$V_i$};
	\node  [black] at (7.0,4.2) {$W$};
	\node  [black] at (5.1,4.4) {$\varphi$};
	\node  [black] at (6.9,3.1) {$\pi_i$};
	\node  [black] at (7.0,1.9) {$W_i$};
	\node  [black] at (3.3,3.1) {$\pi_i$};
	\node  [black] at (5.1,1.7) {$\varphi_i$};
\end{tikzpicture}
\end{center}
\end{figure}

How can we, when each $\varphi_i$ is a homomorphism 
$\varphi_i : G_i \rightarrow H_i$, ensure that such a function $\varphi$ is 
also a homomorphism $\varphi : {\mathbf P}({\cal G}) \rightarrow {\mathbf P}({\cal H})$?  
Let us call a product process $\mathbf{P}$ \textit{hom-preserving} 
if for each index set $I$, every choice of two sets of graphs 
${\cal G}$ and ${\cal H}$, as well as each set $\Phi = \{ \varphi_i \mid i \in I \}$ of
homomorphisms, the function $\varphi$ is a homomorphism 
from ${\mathbf P}({\cal G})$ to ${\mathbf P}({\cal H})$ too. 
It seems reasonable to name the homomorphism $\varphi$ 
resulting in such a way from a hom-preserving process 
$\mathbf P$ the \textit{$\mathbf P$-tensor product of} $\Phi$ 
and write $\varphi = \mathbf{\large{\bigcirc\!\!\!\!\! \times}}_{\mathbf P}\Phi$. 
Section 6 takes this further by linking it to the tensor product 
of graphs associated with the given homomorphisms.  

\textbf{Example:}  
Suppose an index set $I$, two sets of graphs ${\cal G}$ and ${\cal H}$ 
and a set $\Phi$ of homomorphisms as above are given.  
Let $\mathbf{P}$ be the process through which the 
direct product of a set of graphs is formed.  
Then $\mathbf{P}$ is hom-preserving:   \\ 
If $fg \in E_{{\times}({\cal G})}$, the edge set of {\Large$\times$}$_{i \in I} G_i$, 
then $f(i)g(i) \in E(G_i)$ for every $i \in I$.     
Hence, since each $\varphi_i$ is (given as) a homomorphism,  
$\varphi_i(f(i))\varphi_i(g(i)) \in E(H_i)$ for every $i \in I$.  
The defining conditions of $\varphi$ (which do not require 
projections $\pi_i$ to be homomorphisms) allow us to conclude that 
$\varphi(f)(i)\varphi(g)(i) \in E(H_i)$ for every $i \in I$.
But this means that $\varphi(f)\varphi(g) \in E_{{\times}({\cal H})}$, 
the edge set of {\Large$\times$}$_{i \in I} H_i$,  
by the choice of $\mathbf{P}$.  
Hence $\mathbf P \!$  = {\Large$\! \times$} is a hom-preserving process.   

We now generalize the above example to a class of 
constraint-determined product processes. 
\vspace{.165cm}

\begin{thm}\label{Cd is HP} 
Every constraint-determined product processes $\mathbf P$ 
for which one of the $\mathbf P$-constraints is that 
$L(fg) = \emptyset$ for all $fg \in E{\mathbf P}$ 
(briefly: $L = \emptyset$, as for {\raisebox{-0.56mm}{\Large \bf $\Box$}},  
{\Large$\times$}, and 
{\raisebox{-.3ex}{\Large\bf$\Box$}\hspace{-3.9mm}{\Large${\times}$}}) 
is hom-preserving.
\end{thm} 

\begin{proof}
Consider a constraint-determined product processes $\mathbf P$ 
and the function $\varphi: V \rightarrow W$ as 
determined by $\Phi = \{ \varphi_i \mid i \in I \}$, 
$\varphi_i : G_i \rightarrow H_i$.  
We need to prove that $\varphi : {\mathbf P}({\cal G}) \rightarrow {\mathbf P}({\cal H})$. 
Consider all $fg \in E_{{\mathbf P}({\cal G})}$.  
This is equivalent to considering all the triples 
$J(fg), K(fg), L(fg)$ of subsets of $I$ satisfying the $\mathbf P$-constraints -- 
which determine $ E_{{\mathbf P}({\cal G})}$. 
We now assume that always $L(fg) = \emptyset$, and hence 
$I = J(fg) \sqcup K(fg)$ for every $fg \in E{\mathbf P}$.  
Pick any $fg \in E_{{\mathbf P}({\cal G})}$, $f,g \in V$.  
If $j \in J(fg)$, then $f(j)g(j) \in E(G_j)$ and 
(since $\varphi_j : G_j \rightarrow H_j$) 
$\varphi_j(f(j))\varphi_j(g(j)) = \varphi(f)(j)\varphi(g)(j)\in E(H_j)$, 
ensuring that $j \in J(\varphi(f)\varphi(g))$.  
And if $k \in K(fg)$, $f(k) = g(k)$ in  $V(G_k)$ and hence  
(since $\varphi_k : V(G_k) \rightarrow V(H_k)$ is a function) 
$\varphi_k(f(k)) = \varphi_k(g(k))$, i.e. 
$\varphi(f)(k) = \varphi(g)(k)$ in $V(H_k)$, 
ensuring that $k \in K(\varphi(f)\varphi(g))$.  
This means that the two index sets $J(fg)$ and $K(fg)$   
are respectively equal to the two index sets $J(\varphi(f)\varphi(g))$ 
and $K(\varphi(f)\varphi(g))$ determining $\varphi(f)\varphi(g) 
\in E_{{\mathbf P}({\cal H})}$.

This confirms that $\varphi$ is a homomorphism and thus that 
$\mathbf P$ is hom-preserving:  $\varphi =  
\mathbf{\large{\bigcirc\!\!\!\!\! \times}}_{\mathbf P}\Phi
: {\mathbf P}({\cal G}) \rightarrow {\mathbf P}({\cal H})$.
\end{proof}

When $\mathbf P$ is constraint-determined with the index set 
$L(fg)$ empty for all $fg \in E{\mathbf P}$, 
and hence hom-preserving with $\varphi =  
\mathbf{\large{\bigcirc\!\!\!\!\! \times}}_{\mathbf P}\Phi$ 
indeed a homomorphism, we may link the commutative square of 
this section to the procedure in the previous section 
of transforming each $G_i$ to ${\mathbf P}(G_i) = G_i / \theta_i$ (and each 
$H_i$ to ${\mathbf P}(H_i) = H_i / \theta_i$) in order to make each projection 
$\pi_i$ a homomorphism.  
(So far in this section $\pi_i$ need not be a homomorphism at all.) 
Consider the commutative square of functions in which $\varphi$ 
and the two projections $\pi_i$ are now homomorphisms. 
It gives satisfaction that this is indeed now a 
commutative square of homomorphisms, as we now show.

\begin{figure}[ht]
\label{fig:network}
\begin{center}
\begin{tikzpicture}
	\draw [->, thick] (4.1,4.2) -- (6.1,4.2);
	\draw [->, thick] (6.5,4.0) -- (6.5,2.2);
    \draw [->, thick] (3.7,4.0) -- (3.7,2.2);
    \draw [->, thick] (4.1,2.0) -- (6.1,2.0);
	\node  [black] at (3.3,4.2) {${\mathbf P}({\cal G})$};
	\node  [black] at (3.4,1.9) {${\mathbf P}(G_i)$};
	\node  [black] at (7.0,4.2) {${\mathbf P}({\cal H})$};
	\node  [black] at (5.1,4.4) {$\varphi$};
	\node  [black] at (6.9,3.1) {$\pi_i$};
	\node  [black] at (7.0,1.9) {${\mathbf P}(H_i)$};
	\node  [black] at (3.3,3.1) {$\pi_i$};
	\node  [black] at (5.1,1.7) {$\varphi_i$};
\end{tikzpicture}
\end{center}
\end{figure}
\noindent 

\begin{lem}\label{Lemma 6}  
$\varphi_i : {\mathbf P}(G_i) \rightarrow {\mathbf P}(H_i)$.   
\end{lem} 

\begin{proof} 
We know that $\varphi_i : G_i \rightarrow H_i$.  
Suppose that $xy \in \widehat{E_i}(G_i)$, then there are two 
possibilities:  \\ 
(i) $xy \in E_i(G_i)$ and, since $\varphi_i : G_i \rightarrow H_i$, we have 
$\varphi_i(x)\varphi_i(y) \in E_i(H_i) \subseteq \widehat{E_i}(H_i)$; or   \\ 
(ii) there exist $f,g \in V_{{\mathbf P}({\cal G})}$ such that 
$fg \in E_{{\mathbf P}({\cal G})}$, $f(i) = x$, $g(i) = y$ and $i \in J(fg)$.  
As in the proof (for $\mathbf P$) of Theorem \ref{Cd is HP}, it is clear that 
$J(fg) = J(\varphi(f)\varphi(g))$ and $K(fg) = K(\varphi(f)\varphi(g))$ 
(and $L$ always empty).  
Hence $i \in J(\varphi(f)\varphi(g))$ and $\varphi(f)(i)\varphi(g)(i) = 
\varphi_i(f(i))\varphi_i(g(i)) = \varphi_i(x)\varphi_i(y) \in \widehat{E_i}(H_i)$.  
\end{proof}

\section{Graphs, products, and homomorphisms intertwined}

\noindent 
We start this section by constructing a graph from 
any homomorphism.  
Consider a homomorphism   $\varphi : G \rightarrow H$. 
We define a graph $\Gamma(\varphi)$, called \textit{the graph of} $\varphi$, 
by stipulating that  \\
$V_{\Gamma(\varphi)} : = \varphi = \{ (x, \varphi(x)) \mid x \in V_G \}$  
$(\subseteq V_G \times V_H)$; while  \\   
$E_{\Gamma(\varphi)} : = \{ (x,\varphi(x))(y, \varphi(y)) \mid x,y \in V_G   
\mbox{ and } \varphi(x)\varphi(y) \in E_H \}$. 

Note that $\Gamma(\varphi)$ is a simple graph.  
Also, in terms of the two projection mappings $\pi_1: (x, \varphi(x)) \mapsto x$ 
and $\pi_2: (x, \varphi(x)) \mapsto \varphi(x)$, the above definition ensures that  
$\pi_1$ is surjective and $\pi_2$ is a homomorphism  
$\Gamma(\varphi) \rightarrow H$.  
Furthermore, it is clear that of all the information in the 
configuration $\varphi : G \rightarrow H$, the set-theoretical structure 
$\Gamma(\varphi)$ encodes $V_G$, $\varphi$, and $H[\varphi(V_G)]$, but that $E_G$ and 
the rest of $H$ is irretrievable from $\Gamma(\varphi)$.  
   
The assumptions for our last result are now stipulated.   
We have a fixed index set $I$ and two $I$-indexed sets of graphs ${\cal G} = 
\{ G_i \mid i \in I \}$ and ${\cal H} = \{ H_i \mid i \in I \}$.    
$\Phi = \{ \varphi_i \mid i \in I \}$ is a set of homomorphisms 
$\varphi_i : G_i \rightarrow H_i$, while $\mathbf P$ is a constraint-determined 
product process with $L = \emptyset$ as one of the $\mathbf P$-constraints.  
Theorem \ref{Cd is HP} in the previous section then assures us that 
$\varphi = \mathbf{\large{\bigcirc\!\!\!\!\! \times}}_{\mathbf P}\Phi : 
{\mathbf P}({\cal G}) \rightarrow {\mathbf P}({\cal H})$, 
i.e., $\mathbf P$ is hom-preserving.  
We use exactly the same notation as in the previous section for this framework.    
\vspace{.165cm} 

\begin{thm}\label{Final} 
If the configuration $I, {\cal G}, {\cal H}, \Phi, {\mathbf P}$ 
satisfies the conditions just stipulated, then   \\
$\Gamma(\mathbf{\large{\bigcirc\!\!\!\!\! \times}}_{\mathbf P}
\{ \varphi_i \mid i \in I \})  \cong 
{\mathbf P}(\{ \Gamma(\varphi_i) \mid i \in I \})$.   
\end{thm} 

\begin{proof} 
The vertex set of the graph $\Gamma(\mathbf{\large{\bigcirc\!\!\!\!\! \times}}_{\mathbf P}
\{ \varphi_i \mid i \in I \}) = \Gamma(\varphi)$ 
is the function $\varphi : V \rightarrow W$,  
and the vertex set of the graph ${\mathbf P}(\{ \Gamma(\varphi_i) \mid i \in I \})$ 
is the Cartesian product of the vertex sets of the graphs $\Gamma(\varphi_i)$, i.e., 
$\Pi_{i \in I}\varphi_i$.  
Consider the mapping $\alpha : \varphi \rightarrow \Pi_{i \in I}\varphi_i$ 
defined by $(f,\varphi(f)) \mapsto f^\star$ for every $ (f,\varphi(f)) \in \varphi$, where 
$f^\star : I \rightarrow \bigsqcup_{i \in I} \varphi_i$ with   
$$f^\star(i) : = (f(i), \varphi(f)(i)) = (f(i), \varphi_i(f(i)))\in \varphi_i 
\mbox{ for every } i \in I.$$

The function $\alpha$ is easily seen to be a bijection:  \\  
$\bullet$  If $(f, \varphi(f)) \neq (g, \varphi(g))$, then  
$f \neq g$ and hence $f(i) \neq g(i)$ for some $i \in I$.  
But then $f^\star \neq g^\star$.   \\  
$\bullet$  Given any function $h \in \Pi_{i \in I}\varphi_i$, 
one can describe a function $f$ for which $(f, \varphi(f)) \in \varphi$ 
and $f^\star = h$ by using the defining equation above backwards.   

Now consider any two vertices $(f,\varphi(f)), (g,\varphi(g)) 
\in \varphi = V_{\Gamma(\varphi)}$.  
By the definition of adjacency of vertices in $\Gamma(\varphi)$, 
$$(f,\varphi(f))(g,\varphi(g)) \in E_{\Gamma(\varphi)} \mbox{ iff } 
\varphi(f)\varphi(g) \in E({\mathbf P}(\{ H_i \mid i \in I \})). $$ 
With respect to the product ${\mathbf P}(\{ H_i \mid i \in I \})$, 
for every $j \in I$, 
\begin{eqnarray}
j \in J\{ \varphi(f),\varphi(g) \} & \mbox{ iff } &  \varphi(f)(j)\varphi(g)(j)  \in E(H_j) \nonumber \\ 
                    &   \mbox{ iff }  &   \varphi_j(f(j))\varphi_j(g(j)) \in E(H_j)  \nonumber  \\ 
        &   \mbox{ iff }  & (f(j),\varphi_j(f(j)))(g(j),\varphi_j(g(j))) \in E(\Gamma(\varphi_j))  \nonumber  \\ 
        &   \mbox{ iff }  &   f^\star(j)g^\star(j) \in E(\Gamma(\varphi_j))  \nonumber  \\ 
                    &   \mbox{ iff }  &   j \in J\{ f^\star, g^\star \},  \nonumber  
\end{eqnarray} 
the latter with respect to the product ${\mathbf P}(\{ \Gamma(\varphi_i) \mid i \in I \})$.  
So, as subsets of $I$, $J\{ \varphi(f),\varphi(g) \}  = J\{ f^\star, g^\star \}$.  
According to the assumed properties of $\mathbf P$, 
in particular since $L = \emptyset$, then also 
$K\{ \varphi(f),\varphi(g) \}  = K\{ f^\star, g^\star \}$, 
and hence $\varphi(f)\varphi(g) \in E({\mathbf P}(\{ H_i \mid i \in I \}))$ 
iff $f^\star g^\star \in E({\mathbf P}(\{ \Gamma(\varphi_i) \mid i \in I \}))$.  
From what was said previously it follows that 
$(f,\varphi(f))(g,\varphi(g)) \in E_{\Gamma(\varphi)}$ iff 
$f^\star g^\star \in E({\mathbf P}(\{ \Gamma(\varphi_i) \mid i \in I \}))$.  
 
Hence $\alpha$ is an isomorphism and  
$\Gamma(\varphi) \cong {\mathbf P}(\{ \Gamma(\varphi_i) \mid i \in I \})$ 
now follows.  
\end{proof}

\section{The analogy to linear algebra}

\noindent 
We now return to redeem an early promise. 
In what sense could we claim that ``$\mathbf{P}$-morphisms in graph theory 
are analogous to multilinear mappings in linear algebra''? 
To clarify the analogy we display the correspondences in two 
parallel columns:   \\[2mm]
\begin{tabular}{p{6.5cm}}
\textbf{Graph Theory}  \\[0.8mm]  
We have a set ${\cal G} = \{ G_i \}_{i \in I}$ of graphs $G_i$ 
and form the Cartesian product $V = \Pi_{i \in I} V_i$ 
of the underlying sets $V_i$ of (vertices of) the graphs $G_i$. 
One may consider arbitrary graphs $H$ and arbitrary  
$\mathbf{P}$-morphisms $\delta : V \stackrel{\mathbf{P}}{\rightarrow} H$. 
\end{tabular} 
\hfill 
\begin{tabular}{p{6.5cm}}  \\[-7.1mm]
\textbf{Linear Algebra}  \\[0.8mm] 
We have a set ${\cal G} = \{ G_i \}_{i \in I}$ of vector spaces $G_i$ 
and form the Cartesian product $V = \Pi_{i \in I} V_i$ 
of the underlying sets $V_i$ of (vectors of) the vector spaces $G_i$. 
One may consider arbitrary vector spaces $H$ and arbitrary  
multilinear mappings $\delta : V \rightarrow H$. 
\end{tabular}

\noindent
\begin{tabular}{p{6.5cm}}
The canonical $\mathbf{P}$-tensor product of $\cal G$ is the pair 
$(id_V, {\mathbf P}({\cal G}))$ in which $id_V$ is the fixed identity 
$id_V: V \stackrel{\mathbf P}{\rightarrow}{\mathbf P}({\cal G})$, 
satisfying the defining tensor conditions:  \\   
(i) $id_V$ is surjective; and   \\  
(ii) if $H$ is any graph and $\delta : V \stackrel{\mathbf{P}}{\rightarrow} H$ 
any $\mathbf{P}$-morphism, then the function $\delta$ 
itself is a graph homomorphism which we 
 may write as 
$\delta^\star: {\mathbf P}({\cal G}) \rightarrow H$, -- and, 
of course, $\delta = \delta^\star \circ id_V$. 
\end{tabular} 
\hfill 
\begin{tabular}{p{6.5cm}}  \\[-6.3mm]
A tensor product of $\cal G$ is a pair 
$(\varphi, T)$ in which $\varphi$ is a fixed multilinear mapping  
from $V$ to the vector space $T$, 
satisfying the defining tensor conditions:  \\ 
(i) $\varphi(V)$ spans $T$; and    \\  
(ii) if $H$ is any vector space and $\delta : V \rightarrow H$ 
any multilinear mapping, then there exists a linear mapping 
(i.e., vector space homomorphism) 
$\delta^\star: T \rightarrow H$ such that $\delta = \delta^\star \circ \varphi$. 
\end{tabular}

\section{Acknowledgement}  

\noindent
Izak Broere is supported in part by the National Research 
Foundation of South Africa (Grant Number 90841).



\begin{thebibliography}{99}

\bibitem{BHP15}
{Broere, I., Heidema, J. and Pretorius, L.M.:} 
Graph congruences and what they connote. 
Submitted (2015)  

\bibitem{Di10} 
{Diestel, R.:} 
Graph Theory, Fourth Edition. 
Graduate Texts in Mathematics 173, 
Springer, Heidelberg (2010) 

\bibitem{HIK11}
{Hammack, R., Imrich, W., Klav\v{z}ar, S.:} 
Handbook of Product Graphs, Second Edition. 
CRC Press, Boca Raton (2011) 

\bibitem{Ro92} 
{Roman, S.:} 
Advanced Linear Algebra. 
Graduate Texts in Mathematics 135, 
Springer, Heidelberg (1992) 

\bibitem{Ro08} 
{Rosinger, E.E.:} 
Two Generalizations of Tensor Products, Beyond Vector Spaces. 
arXiv:0807.1436v4 (2008) 

\end{thebibliography}
\end{document}